\documentclass[12pt, english]{article}

\usepackage{babel}
\usepackage[latin1]{inputenc}
\usepackage{amsmath}
\usepackage{amssymb}
\usepackage{amsthm}

\theoremstyle{plain}
\newtheorem{Thm}{Theorem}
\newtheorem{Lemma}[Thm]{Lemma}

\theoremstyle{definition}
\newtheorem*{Strategy*}{Strategy}

\begin{document}

\title{Partially ordered secretaries}

\author{Ragnar Freij and Johan W\"astlund \\
\small Department of Mathematics\\[-0.8ex]
\small Chalmers University of Technology, \\[-0.8ex] 
\small S-412 96 G\"oteborg, Sweden\\[-0.8ex]
\small \texttt{freij@chalmers.se, wastlund@chalmers.se}
}
\date{\small \today}

\maketitle

\begin{abstract}  The elements of a finite nonempty partially ordered set are exposed at independent uniform times in $[0,1]$ to a selector who, at any given time, can see the structure of the induced partial order on the exposed elements. The selector's task is to choose online a maximal element. 

This generalizes the classical linear order secretary problem, for which it is known that the selector can succeed with probability $1/e$ and that this is best possible.

We describe a strategy for the general problem that achieves success probability at least $1/e$ for an arbitrary partial order.
\end{abstract}

\section{Background}
The classical secretary problem asks for a strategy that with reasonable probability picks online the best of $n$ applicants for a job, given that at any time the only information available is the relative ranks of the applicants that have been interviewed so far. For a historical overview of the secretary problem and some of its generalizations, consult \cite{F89}.

The well-known solution is to reject the first $n/e$ applicants (approximately), and after that to accept the first applicant who is better than all those. The probability of success is asymptotically $1/e$ for large $n$, and this is best possible. 

It was observed by J.~Preater \cite{P99} that the selector can achieve success probability bounded away from zero in a partially ordered version of the problem. Here the secretaries are replaced by the elements of a partially ordered set $P$. These elements are exposed in random order, and at every time the selector can see the order relations between all pairs of exposed elements, in other words the induced partial order on the exposed elements. The task is now to select online one of the maximal elements of $P$.

The details of the problem can be phrased in a couple of different ways. The selector can be assumed to know in advance the number of elements in $P$. Alternatively, the selector knows the structure of $P$, but cannot ``recognize'' the individual elements as they arrive. In this later case, optimal results are known for some special classes of partial orders. See, for instance, \cite{M98} for a strategy that is optimal when $P$ is known to be a binary tree.

In this paper we consider a continuous time version which is at least as hard for the selector as these discrete time versions, and which can easily be shown to be equivalent as far as the best general success probability goes. In this version, the elements of $P$ are exposed at independent uniform times in the interval $[0,1]$. We ask for a strategy that achieves success probability at least $c>0$ for every finite nonempty partial order $P$. Obviously such a strategy for the continuous time version can be applied also in discrete time, since knowing the number $n$ of elements in $P$, we can generate $n$ random ``times'' in $[0,1]$ ourselves, and assign them in order to the elements as they arrive.
 
In \cite{P99}, Preater described a simple strategy and showed that it succeeds with probability at least $1/8$ for every $P$. The analysis of Preater's strategy was refined by N.~Georgiou, M.~Kuchta, M.~Morayne and J.~Niemiec \cite{GKMN08}, who showed that with a trivial improvement, it actually succeeds with probability at least $1/4$.

Later, Kozik \cite{K10} suggested another strategy for the general situation, and proved that it has success probability at least a constant $c>1/4$ for every sufficiently large $n$. Kozik's strategy is indeed a generalization of the strategy in the linear case. 

Here, we describe a different strategy that achieves success probability at least $1/e$ for every $P$. Since this matches the upper bound given by the analysis of the classical linear order case, it settles the question of the best possible general success probability. 

\section{Description and analysis of the strategy}
Suppose that the elements of a finite nonempty partial order $P$ are assigned distinct real \emph{weights}. We define the \emph{greedy maximum} of $P$ as follows. Let $z_0$ be the element of smallest weight in $P$. As long as $z_i$ is not maximal, let $z_{i+1}$ be the element of minimum weight among the elements larger than $z_i$. This gives a chain whose terminal element is the greedy maximum of $P$.

Our strategy for the partially ordered secretary problem is as follows:
\begin{Strategy*}
We assign independent uniform weights from $[0,1]$ to the elements as they arrive. After rejecting everything up to time $1/e$, we accept the first element $x$ which is itself the greedy maximum of the induced partial order $P_x$ on the elements, including $x$, that have been exposed up to the time when $x$ arrives. 
\end{Strategy*}

\begin{Thm} \label{T:main}
This strategy succeeds in accepting a maximal element with probability at least $1/e$ for every partially ordered set.
\end{Thm}

The rest of the paper is devoted to the proof of Theorem~\ref{T:main}. To simplify the discussion, let us say that we \emph{tag} the element $x$ if it is the greedy maximum of $P_x$. In order for our strategy to \emph{accept} $x$, three things are thus required: $x$ must arrive at a time $t>1/e$, it must be tagged, and finally no other element must be tagged between time $1/e$ and time $t$. 

In the following, we think of $P$ as a fixed partially ordered set of $n$ elements.

\begin{Lemma} \label{L:independence}
Let $a_1,\dots, a_n$ be the elements of $P$ in the order that they are exposed. Let $A_k$ be the event that $a_k$ is tagged. Then $Pr(A_k) = 1/k$ and $A_1,\dots, A_n$ are independent. 
\end{Lemma}

\begin{proof}
Suppose that we know $P$ and the weights that are eventually assigned to all its elements. Suppose moreover that we know the elements $a_{k+1},\dots, a_n$. Then in particular we know which of them are tagged. Now consider the greedy maximum of the induced partial order on $\{a_1,\dots, a_k\}$. The probability that this element was the last one to arrive and thereby labeled $a_k$ is clearly $1/k$. This shows that $Pr(A_k) = 1/k$ conditioning on $A_{k+1},\dots, A_n$, establishing our claim.
\end{proof}

\begin{Lemma} \label{L:uniform}
Let $0\leq t\leq 1$. Conditioning on the set of elements being exposed before time $t$, and assuming that this set is nonempty, the arrival time of the last element to be tagged before time $t$ is uniform in the interval $[0,t]$.
\end{Lemma}

\begin{proof}
Suppose $k$ elements arrive before time $t$. Lemma~\ref{L:independence} shows that the joint distribution of $A_1,\dots, A_k$ is the same regardless of the structure of the induced partial order on $\{a_1,\dots,a_k\}$. Therefore it suffices to establish the claim for the linear order on $k$ elements. For this particular order, it is obvious that the last element to be tagged is the unique maximal element, and its arrival time is uniform in $[0,t]$.
\end{proof}

Supposing $P$ is fixed, but the weights of its elements independent and uniform in $[0,1]$, let $\mu(x)$ be the probability that $x$ is the greedy maximum of $P$. Moreover, for $0\leq t \leq 1$, let $\mu_t(x)$ be the probability that $x$ is the greedy maximum of $P$, conditioning on the weight of $x$ being at most $t$.

\begin{Lemma} \label{L:mu_t}
Suppose that $x$ is maximal in $P$. Conditioning on $x$ arriving at time $t$, $$Pr(\text{$x$ is tagged})  = \mu_t(x).$$
\end{Lemma}

\begin{proof} 
Assign weights to the elements of $P$ and condition on $x$ having weight at most $t$. Since elements of weight greater than $t$ (and therefore greater than the weight of $x$) will have no influence on whether or not the greedy chain terminates at $x$, $\mu_t(x)$ is the probability that $x$ is the greedy maximum in a partial order obtained by first discarding all elements except $x$ independently with probability $1-t$, and then choosing all weights uniformly in $[0,t]$. 

Two observations conclude the proof: First, if all weights are chosen from $[0,t]$, we might as well choose them from $[0,1]$. Second, independently discarding all elements except $x$ with probability $1-t$ is precisely how we get $P_x$ conditioning on $x$ arriving at time $t$.
\end{proof}

\begin{Lemma} \label{L:tag}
Conditioning on $x$ arriving at time $t$, $$Pr(\text{$x$ is tagged}) \geq \mu(x).$$
\end{Lemma}

\begin{proof}
In view of Lemma~\ref{L:mu_t}, it suffices to notice that since decreasing the weight of an element can only increase its probability of being the terminal point of the greedy chain, $$\mu_t(x) \geq \mu(x).$$
\end{proof}

\begin{proof}[Proof of Theorem~\ref{T:main}]
To prove Theorem~\ref{T:main}, only a simple calculation remains. Let $x$ be maximal in $P$, and suppose that $x$ arrives at time $t>1/e$. We want to estimate the probability that $x$ is accepted as a function of $t$ and $\mu(x)$. 

By Lemma~\ref{L:tag}, the probability that $x$ is tagged is at least $\mu(x)$. If we condition on $P_x$, then either $P_x - \{x\}$ is empty, which means $x$ was the first tagged element, or it is nonempty, and by Lemma~\ref{L:uniform} the arrival time of the last element that was tagged before time $t$ is uniform in $[0,t]$. In any case, the probability that no element was tagged between time $1/e$ and time $t$ is at least $1/(et)$.
Therefore the total probability that $x$ is accepted is at least $$\mu(x) \cdot \int_{1/e}^1 \frac1{et}\, dt = \frac1e  \cdot  \mu(x).$$
Summing over all maximal elements of $P$ we find that the probability that one of them is accepted is at least $1/e$.
\end{proof}

\end{document}